\documentclass[preprint,12pt]{elsarticle}

\usepackage{amsfonts,amsmath,amsthm,amssymb}
\usepackage{pdfsync}

\usepackage{graphicx}
\usepackage{color}

\usepackage{pstricks}
\usepackage{pst-plot}

\newcommand\cns{{\mathsf k}}
\newcommand\bil{{\mathcal B}}

\newcommand\R{\mathbb{R}}
\renewcommand\tt{\boldsymbol{t}}

\newcommand\nn{\boldsymbol{n}}
\newcommand\etab{\boldsymbol{\eta}}
\newcommand\betab{\boldsymbol{\beta}}
\newcommand\taub{\boldsymbol{\tau}}

\newcommand\gammab{\boldsymbol{\gamma}}
\newcommand\esseb{\boldsymbol{s}}

\newcommand\w{w}          % deflections
\newcommand\W{W}      % space of deflections
\newcommand\rots{\boldsymbol{\theta}}  % rotations
\newcommand\Rots{\boldsymbol{\Theta}}  %space of rotations
\newcommand\shears{\gammab}  % shears
\newcommand\Shears{{\boldsymbol{\Sigma}}} % spazio degli shears
\newcommand\Gammaref{\widehat \Gamma} % bordo parametrico

\def\R{\mathbb R}

\def\bv{\,{\bf v}\,}

\def\bF{\,{\bf F}\,}

\def\ba{{\boldsymbol{\alpha}}}

\newcommand{\Wh}{{W}_{h,0}}
\newcommand{\Whref}{\widehat{W}_{h,0}}
\newcommand{\WhrefF}{\widehat{W}_{h}}
\newcommand{\Rotsh}{{\Rots}_{h,0}}
\newcommand{\Rotshref}{\widehat{\Rots}_{h,0}}
\newcommand{\RotshrefF}{\widehat{\Rots}_{h}}
\newcommand{\Shearsh}{{\Shears}_{h,0}}
\newcommand{\Shearshref}{\widehat{\Shears}_{h,0}}
\newcommand{\ShearshrefF}{\widehat{\Shears}_{h}}
\newcommand{\Shearshtilde}{\widetilde{\Shears}_{h,0}}

\newcommand{\cC}{\mathcal{C}}

\newcommand{\Omegaref}{\widehat {\Omega}}

\renewcommand{\ba}{\alpha}

\newtheorem{prop}{Proposition}[section]
\newtheorem{rem}{Remark}[section]
\newtheorem{lemma}{Lemma}[section]
\newtheorem{corol}{Corollary}[section]
\newtheorem{assum}{Assumption}[section]
\def\trait #1 #2 #3 {\vrule width #1pt height #2pt depth #3pt}
\def\fin{\hfill
        \trait .3 5 0
        \trait 5 .3 0
        \kern-5pt
        \trait 5 5 -4.7
        \trait 0.3 5 0
\medskip}

\date{}

\journal{Computer Methods in Applied Mechanics and Engineering}

\begin{document}

\begin{frontmatter}

\title{An isogeometric method for the Reissner-Mindlin plate bending problem}

\author[unimi]{L.~Beir\~ao~da~Veiga}
\ead{lourenco.beirao@unimi.it}

\author[imati]{A.~Buffa}
\ead{annalisa@imati.cnr.it}

\author[unipv]{C.~Lovadina}
\ead{carlo.lovadina@unipv.it}

\author[unipv]{M.~Martinelli\fnref{cor}}
\ead{martinelli@imati.cnr.it}

\author[unipv]{G.~Sangalli}
\ead{giancarlo.sangalli@unipv.it}

\fntext[cor]{Corresponding author}

\address[unimi]{Dipartimento di Matematica ``F.~Enriques", Universit\`a degli Studi di Milano, Via Saldini 50, 20133 Milano, Italy}
\address[imati]{IMATI - CNR, Via Ferrata 1, 27100 Pavia, Italy}
\address[unipv]{Dipartimento di Matematica ``F.~Casorati", Universit\`a di Pavia, Via Ferrata 1, 27100 Pavia, Italy}

\begin{abstract}
We present a new isogeometric method for the discretization of the Reissner-Mindlin plate bending problem. 
The proposed scheme follows a recent theoretical framework that makes possible to construct a space
of smooth discrete deflections $W_h$ and a space of smooth discrete rotations $\Rots_h$ such that the Kirchhoff contstraint is exactly satisfied at the limit.
Therefore we obtain a formulation which is natural  from the theoretical/mechanical viewpoint and locking free by construction.
We prove that the method is uniformly stable and satisfies optimal convergence estimates. Finally, the theoretical results are fully supported by numerical tests.
\end{abstract}


\end{frontmatter}
%\maketitle

%---------------------------------------------------------------------------
\section{Introduction}\label{introduction}
%---------------------------------------------------------------------------

The Reissner-Mindlin theory is widely used to describe the bending
behavior of an elastic plate loaded by a transverse force. Despite
its simple formulation, the discretization by means of finite
elements is not straightforward, since standard low-order schemes
exhibit a severe lack of convergence whenever the thickness is too
small with respect to the other characteristic dimensions of the
plate.  This undesirable phenomenon, known as {\em shear locking},
is nowadays well understood: as the plate thickness tends to zero,
the Reissner--Mindlin model enforces the Kirchhoff constraint,
which is typically too severe for Finite Element Methods (FEM), especially if low-order polynomials are employed (see, for
instance, the monograph by Brezzi and Fortin~\cite{brefor}). 
Roughly speaking, the root of the shear locking phenomenon is that {\em the space of discrete functions which satisfy the Kirchhoff 
constraint is very small, and does not properly approximate a generic plate solution}. 
The most popular way to overcome the shear locking phenomenon in FEM is to
reduce the influence of the shear energy by considering a mixed
formulation and/or suitable {\em shear reduction operator}. As a consequence, the choice of the discrete spaces
requires particular care, also because of the possible occurrence
of {\em spurious modes}. A vast engineering and
mathematical literature is devoted to the design and analysis of
plate elements for FEM; we mention here, in a totally non-exhaustive way,
the
works~\cite{arfabas,aulov,Be04,BFS91,chasten,DL92,falktu,HF88,lovadina05,lova,tesshughes}. 

%A common feature of all FEM in  literature is that, in order to overcome
%locking, the shear strain condition 
%\begin{equation}\label{intro:kirch}
%t^2 \shears = k \mu (\rots - \nabla\w) \ , 
%\end{equation}
%is imposed in a relaxed way, where the meaning of the above symbols
%can be found in Section \ref{RMmodel} below. \Rd NON E` MEGLIO DIRE
%QUA? \B This constraint relaxation \Rd can be interpreted \B through a mixed
%formulation, which is equivalent to a projection of the shear strains,
%or adopting directly a reduction operator. 

In 2005, IsoGeometric Analysis (IGA) has been  introduced by
T.J.R. Hughes and co-authors in \cite{HCB05} as a  novel technique for the
discretization of partial differential equations. IGA  is  having a
growing impact on several fields, from 
fluid dynamics \cite{BCHZ,BdFS}, to 
structural mechanics 
\cite{ABLR08,MR2361474,BBHH,LEBEH} and
electromagnetics \cite{BSV,BRSV}.  A comprehensive reference for  IGA is the book \cite{CoHuBa:book}.

IGA methodologies are designed with the aim of improving the
interoperability between numerical simulation of physical phenomena and the
Computer Aided Design (CAD) systems. Indeed, the ultimate goal is to
drastically reduce the error in the representation  of the computational
domain and the re-meshing by the use of the ``exact'' CAD  geometry
directly at the coarsest level of discretization. This is achieved by
using  B-Splines or Non Uniform Rational B-Splines (NURBS) for the geometry
description as well as for the representation of the unknown
fields. The use of Spline  or NURBS functions, together with
isoparametric concepts, results in an extremely successful idea and
paves the way to many new numerical schemes enjoying features that
would be extremely hard  to achieve within a standard FEM.
Splines and NURBS 
offer a flexible set of basis functions for which refinement,
de-refinement,  degree elevation and mesh deformation are very  efficient
(e.g., \cite{
%CoHuRe07, %13/06/2011: Manca in bibliografia (Max)
LEBEH}).
 Beside the fact that one  can directly  treat 
geometries described by Splines and NURBS parametrizations, these functions are
interesting in themselves since they easily allow global smoothness beyond the classical $C^0$-continuity of FEM. 
 This feature has been advantageously exploited in
 recent works, for example \cite{MR2361474},
%\cite{GCBH}, %13/06/2011: Manca in bibliografia (Max)
 and \cite{HRS08},  and studied in
\cite{BBRS}, \cite{EBBH}.

Furthermore,  the intrinsic regularity of the Spline basis functions opened the way
 to completely new discretization schemes for Maxwell equations
 \cite{BSV,BRSV}, as
 well as for other problems, such as,  the Stokes problem
 \cite{BdFS}. These schemes are based on  suitable \emph{smooth}
 approximation of differential forms, verifying a De~Rham diagram  in the spirit of
 \cite{AFW10}.  The proposed theoretical framework can be usefully
 adopted also for  discretizing the Reissner-Mindlin plate system which is the object of the present paper. Indeed,
 it makes possible to construct a space
 of smooth
 discrete deflections $W_h$ as discrete 0-forms and  a space of smooth discrete rotations $\Rots_h$ as discrete 1-forms such that it holds:
\[ \nabla W_h \subseteq \Rots_h. \] 
 This allows us to select approximation spaces containing a subspace which both exactly satisfies the Kirchhoff 
 constraint and has optimal approximation features.
 Therefore, we obtain a simple formulation which is 
natural  from the theoretical/mechanical viewpoint since it is built to be locking free, and it is easy to study.  

Finally, the method inherits all the advantages, mentioned above, which are typical of IGA: the capability to incorporate exactly CAD  geometries, and the flexibility in the choice of the polynomial degree and function regularity. Regarding the first advantage, the importance of reproducing the exact geometry in plate analysis has been underlined very clearly in \cite{BP88}. 

The outline of the paper is the following. Section \ref{RMmodel} is devoted to the description of Reissner-Mindlin model and \ref{sec:splines} provides basics definition of spline spaces. Our discretization method is proposed in Section \ref{sec:disc-spaces} and analysed in Section \ref{sec:conv}. Finally, Section \ref{sec:nums} is devoted to the numerical validation.

%---------------------------------------------------------------------------
\section{The Reissner-Mindlin plate bending problem}\label{RMmodel}
%---------------------------------------------------------------------------

{ Let $\Omega $ be a Lipschitz bounded open set of 
$\mathbb{R}^2$  representing the midsurface of the plate and let $\Gamma$ be its boundary.  Other assumptions on the domain $\Omega$ will be set at the end of Section \ref{sec:bspline2d}. We assume
that the boundary $\Gamma$  is the union of three
disjoint sets
$$
\overline{\Gamma} = \overline{\Gamma}_f \cup \overline{\Gamma}_s \cup \overline{\Gamma}_c \ ,
$$}
with $\Gamma_f,\Gamma_s,\Gamma_c$ being a finite union of
connected components. The plate is clamped in $\Gamma_c$, simply
supported in $\Gamma_s$ and free in $\Gamma_f$. We assume for
simplicity of exposition that all boundary conditions are
homogeneous, and that the union $\Gamma_s \cup \Gamma_c$ has
positive measure, in order to neglect rigid body motions of the
plate. Finally, let $\Gamma_s$ be divided into a soft and a hard
part (both being a finite union of connected components), i.e.
$\Gamma_s = \Gamma_{ss} \cup \Gamma_{sh}$. We then set the spaces
for deflections and rotations 
\begin{equation*}%\label{eq:W-Rots-definition}
\begin{aligned}
\W_{0} & = \{ v\in H^1(\Omega) \ : \ v = 0 \ \textrm{on }\ \Gamma_s\cup\Gamma_c \}  \\
\Rots_{0} & = \{ \etab\in [H^1(\Omega)]^2 \ : \ \etab= 0 \ \textrm{on
} \ \Gamma_c \ , \etab\cdot\tt=0 \ \textrm{on } \ \Gamma_{sh} \} ,
\end{aligned}
\end{equation*}
with $\tt$ the unit tangent to $\Gamma$ obtained by an
anti-clockwise rotation of the outward normal $\nn$.

Following the Reissner-Mindlin model, see for instance
\cite{brefor}, the plate bending problem requires to solve
\begin{equation}\label{P}
\left\{
\begin{aligned}
& \text{Find  $\rots\in \Rots_{0}, \w \in
\W_{0}$  such  that} \\
& a(\rots,\etab)+ \mu k t^{-2} (\rots- \nabla \w, \etab-\nabla v)
= (f,v) \quad \forall \etab \in \Rots_{0}, v \in \W_{0} \ ,
\end{aligned}
\right.
\end{equation}
where $\mu$ is the shear modulus  and $k$ is the so-called shear
correction factor. Above, $t$ represents the plate thickness, $\w$
the deflection, $\rots$ the rotation of the normal fibers and $f$
the applied scaled transversal load. Moreover, $( \cdot,\cdot )$ stands
for the standard scalar product in $L^2(\Omega)$ and the bilinear
form $a(\cdot, \cdot)$ is defined by
$$
a(\rots,\etab) = (\mathbb{C} \varepsilon(\rots),
\varepsilon(\etab)),
$$
with $\mathbb{C}$ the positive definite tensor of bending moduli
and $\varepsilon (\cdot)$ the symmetric gradient operator.
Introducing the scaled shear stresses $\shears= \mu k
t^{-2}(\rots- \nabla \w)$, Problem~\eqref{P} can be written in
terms of the following mixed variational formulation:
\begin{equation}\label{Pmixed}
\left\{
\begin{aligned}
& \text{Find  $\rots\in \Rots_{0}, \w \in
\W_{0}, \gammab \in \Shears$  such  that} \\
& a(\rots,\etab)+ (\shears, \etab-\nabla v) = (f,v) \quad \forall
\etab \in \Rots_{0}, v \in \W_{0} \\
&(\rots- \nabla \w,\esseb)- \frac{t^2}{\mu k}(\shears,\esseb)=0
\quad \forall \esseb\in \Shears \ ,
\end{aligned}
\right.
\end{equation}
where $ \Shears = [L^2(\Omega)]^2 $.

The above problem is well posed and $t$-uniformly stable when using 
the $H^1$ norm for the spaces $\Rots_{0}, \W_{0}$, and the norm:
$$
|| \esseb ||_{\Shears} : = t \| \esseb\|_{L^2} +  \sup_{(\etab,v)\in \Rots_0\times\W_0}   \frac{ (\esseb, \etab-\nabla v)}{(||\etab ||_{H^1}^2 + ||v ||_{H^1}^2)^{1/2}}.
$$ 
for the space $\Shears$ (see \cite{brefor}, for instance).
To simplify notation, and without any loss of generality,  we will
assume $\mu k = 1$ in the analysis that follows.

%---------------------------------------------------------------------------
\section{B-splines and piece-wise smooth functions}\label{sec:splines}
%---------------------------------------------------------------------------

\subsection {B-spline spaces and piece-wise smooth functions in one dimension}

Given positive integers $p$ and $n$, such that  $n \ge p + 1$,
we introduce the ordered knot vector
\begin{equation*}
\Xi := \{0 = \xi_1, \xi_2, \dots, \xi_{n+p+1} = 1 \} \,,
%\label{knot_vector}
\end{equation*}
where we allow repetitions of knots, that is, we assume $\xi_1
\leq \xi_2 \leq \dots \leq \xi_{n+p+1}$.
In the following we will only work with {\it open} knot vectors,
which means that the first $p+1$ knots
in $\Xi$ are equal to $0$, and the last $p+1$ are equal to
$1$, and we assume that all internal knots have multiplicity $r$, $1
\le r \le p+1$,  so that
 $$\Xi = \{\underbrace{\zeta_1, \dots, \zeta_{1}
}_{p+1 \text{ times}},\underbrace{\zeta_2, \dots, \zeta_{2} }_{r
  \text{ times}},\ldots, \underbrace{\zeta_m, \dots, \zeta_{m} }_{p+1
  \text{ times}}\} .$$
The  vector $$\mathcal {Z} = \{0=\zeta_1, \zeta_2, \dots, \zeta_m =1\} $$
represents the (ordered) vector of knots without repetitions, and the
relation $  m = \frac{n-p-1}{r} + 2$ holds.

Through the iterative procedure detailed in \cite{HCB05} we
construct $p$-degree (that is, $(p+1)$-order)  \mbox{B-spline} basis functions,
denoted by $B_i$, for
$i=1,\ldots,n$.  These basis functions are piecewise polynomials of
degree $p$ on the subdivision  $ \{\zeta_1, \dots, \zeta_m \} $. At
$\zeta_i$ they have $\alpha := p-r$ continuous derivatives.
Therefore, $-1 \leq \alpha \leq p-1$:  the
maximum multiplicity allowed, $r= p+1$, gives  $\alpha = -1 $,
which  stands for a discontinuity at each $\zeta_i$.
Each basis function $B_i$ is non-negative and supported in the
interval $[\xi_i , \xi_{i+p+1}]$. Moreover, these \mbox{B-spline} functions
constitute a partition of unity, that is
\begin{equation*}
\sum_{i=1}^n B_i(x) = 1 \quad \forall x \in (0,1).
\end{equation*}
The space of \mbox{B-splines} spanned by the basis functions $B_i$ will be
denoted by
\begin{equation*}
S^p_{\ba} := {\rm span} \{ B_i \}_{i=1}^n.
\end{equation*}
%
%An example of quadratic \mbox{B-splines} constructed from
%\Rd the  open knot vector $$\Xi = \{ 0, 0, 0, 1/5, 2/5, 3/5, 3/5, 4/5, 1, 1, 1\} $$
%is presented in \mbox{Figure \ref{fig:bsplines}}. In this case $\ba= \{ -1, 1,1,0,1,-1\} $.
%Notice that, since
%the knot $\xi_6 = \xi_7 = \zeta_4=3/5$ has multiplicity $r_4=2$, the fourth, fifth and sixth
%functions are only continuous ($ \alpha_4 =0$) at that point.
%
Derivatives of splines are splines as well. Let $S^{p+1}_{\alpha+1} $
  and  $S^{p}_{\alpha} $  be spline spaces constructed, according with
  the notation above, on the same subdivision $\{\zeta_1, \dots, \zeta_m \} $.   Then, it is easy to see that
\begin{equation}
  \label{eq:derivative-of-splines}
  \left \{\frac{d}{dx} v : v \in S^{p+1}_{\ba+1} \right \} = S^p_{\ba},
\end{equation}
Notice moreover that
\begin{equation}
  \label{eq:dimension-Spline-p}
  \#S^p_{\ba} = p+1 +(m-2)(p-\alpha),
\end{equation}
and
\begin{equation}
  \label{eq:dimension-Spline-p+1}
  \#S^{p+1}_{\ba+1} = p+2 +(m-2)(p-\alpha),
\end{equation}
where $\#$ is used to denote the dimension of the linear space. Then,
from (\ref{eq:dimension-Spline-p})--(\ref{eq:dimension-Spline-p+1}),
$ \#S^{p+1}_{\ba+1} = \#S^p_{\ba} +1 $, in agreement with  the fact
that the derivative is a surjective operator from $ S^{p+1}_{\ba+1}
$ to  $S^p_{\ba} $ and has a one-dimensional kernel, the constants.

%\begin{figure}[htb]
%\centering
%\includegraphics[width=.9\linewidth]{spline_basis}
%\caption{Quadratic B-splines basis functions constructed from the open knot vector $\Xi=\{ 0,0,0,1,2,3,4,4,1,1,1\}$.}
%\label{fig:bsplines}
%\end{figure}
%
We denote by $\cC^\infty_\ba$ the space of piecewise smooth functions on
%$ \{\zeta_1, \dots, \zeta_m \} $,
$(0,1)$, whose restriction to each subinterval
$(\zeta_i, \zeta_{i+1})$ admits a $C^\infty$ extension to the closed interval
$[ \zeta_i, \zeta_{i+1} ]$ with $\alpha$ continuous derivatives at
$\zeta_i$, for all $i=2,\dots, m-1$.

\subsection {B-spline spaces and piece-wise smooth functions in two dimensions}\label{sec:bspline2d}

The definition of \mbox{B-splines} spaces given above can be extended to two dimensions
as follows.
Let us consider the square $\Omegaref = (0,1)^2 \subset \R^2$, which will be referred to
as \emph{parametric domain}.
Given integers $p_d$, $r_{d}$, $n_d$ and $\alpha_{d} = p_{d}-r_{d}$, with $d=1,2$,
we introduce the knot vectors $\Xi_d = \{\xi_{1,d},
\xi_{2,d},\ldots ,\xi_{ n_d + p_d + 1,d}\} $ and the associated
vectors $ \mathcal {Z} _d= \{\zeta_{1,d}, \dots, \zeta_{m_d,d} \}$ as in the one-dimensional case.
Associated with
these knot vectors there is a {\it mesh} $\Omegaref_h$ of the
parametric domain, that is, a
partition of $(0,1)^2$ into rectangles:
\begin{equation}
  \label{eq:mesh}
  \Omegaref_h=  \{ Q = \otimes_{d=1,2} (\zeta_{i_d,d},\zeta_{i_d+1,d} ), \  1 \leq i_d \leq m_d-1\}.
\end{equation}
Given an element $Q \in   \Omegaref_h$, we set $h_Q =
\mathrm{diam}(Q)$, and define the global mesh size $h=\max\{h_Q,\ Q\in \Omegaref_h\}$.

In this paper we make the following assumption: 

\begin{assum}\label{ass1}
The parametric mesh $\widehat{\Omega}_h$ is shape regular.
\end{assum}

It is important to remark that, due to the tensor product structure, \textbf{ shape regularity implies quasi-uniformity}, i.e., 
 there exists a  positive constant $\cns$,  fixed once and for all, such that
\begin{equation}\label{quasiunif}
\cns h \le h_Q \le h \ , \quad 
\forall Q \in \widehat{\Omega}_h.
\end{equation}

We associate to  the two given  knot vectors $\Xi_d$, $d=1,2$  the  $p_d$-degree univariate \mbox{B-splines} basis
functions    $B_{i,d}$, with $i = 1, \ldots, n_d$. Then,  on the associated mesh $\Omegaref_h $, we define the
tensor-product \mbox{B-spline} basis functions as
\begin{equation*}
B_{ij} := B_{i,1} \otimes B_{j,2}, \quad i=1,\dots,n_1, \; j =
1,\dots,n_2 \,.
\end{equation*}
Then, the tensor product \mbox{B-spline} space is defined as the space spanned by
these basis functions, namely
\begin{equation*}%\label{eq:tensor-product-space}
 S^{p_1,p_2} _{\ba_1,\ba_2}\equiv  S^{p_1,p_2} _{\ba_1,\ba_2}
 (\Omegaref_h):=S^{p_1} _{\ba_1} \otimes S^{p_2} _{\ba_2}  =  {\rm span} \{B_{ij} \}_{i=1,j=1}^{n_1,n_2} \,.
\end{equation*}
Notice that the space $S^{p_1,p_2} _{\ba_1,\ba_2}
(\Omegaref_h) $ is fully characterized by the mesh $\Omegaref_h$,
by $p_1$, $p_2$, $\ba_1$ and $\ba_2$, as our notation reflects.
The minimum regularity of the space is $\alpha := \min \{\alpha_{d}:\,  d = 1,2 \} $.

In a similar way, we define on $\Omegaref_h $ the space of
piecewise smooth functions with interelement regularity on the vertical and  horizontal
mesh edges given by $\alpha_{1}$ and $\alpha_{2}$ respectively.
This is denoted by
\begin{displaymath}
  \cC^\infty_{\ba_1,\ba_2}= \cC^\infty_{\ba_1,\ba_2}(\Omegaref_h)= \cC^\infty_{\ba_1}\otimes \cC^\infty_{\ba_2}.
\end{displaymath}
Precisely, a function in $\cC^\infty_{\ba_1,\ba_2}$ admits a
$C^\infty$ extension in the closure of each element $Q \in
\Omegaref_h$,   has $\alpha_{1}$ continuous derivatives on the
edges $\{ (x_1, x_2) : x_1=\zeta_{i,1}, \zeta_{j,2} < x_2 <
\zeta_{j+1,2} \}$, for $j=1, \ldots, m_2-1$,
$i=2, \ldots, m_1-1$ and $\alpha_{2}$
continuous derivatives on the edges $\{ (x_1, x_2) : \zeta_{j,1} < x_1
< \zeta_{j+1,1}, x_2=\zeta_{i,2},  \}$, for $j=1, \ldots, m_1-1$, $i=2, \ldots, m_2-1$.
From the definitions, $S^{p_1,p_2} _{\ba_1,\ba_2} \subset
\cC^\infty_{\ba_1,\ba_2}$.

From an initial coarse mesh $ \Omegaref_{h_0}$, refinements
are constructed by knot insertion (with possible
repetition, see \cite{DeBoor}). Therefore,  we end up considering a  family of
meshes $\{\Omegaref_h\}_{h \leq h_0}$ and associated spaces,  with  the global
mesh size  $h$ playing the
role of family index, as usual in finite element literature.

\medskip

We assume that our computational domain $\Omega \subset \R^2$ can be exactly
parametrized by a geometrical mapping $\bF : \Omegaref
\longrightarrow  \Omega  $ which belongs to $(\cC^\infty
_{\alpha_1,\alpha_2} (\Omegaref_{h}))^2$, with  piecewise smooth inverse, and is independent
of the mesh family index $h$. 
The geometrical map $\bF$ naturally
induces a mesh $\Omega_h$ on $\Omega$, which is the image of (\ref{eq:mesh}). 
Notice that, as a conseguence of \eqref{quasiunif} and the boundness of $\bF$ and its inverse, there exist two positive constants $\cns'$,
$\cns''$, fixed once and for all, such that
\begin{equation}\label{quasiunif-phys}
\cns' h \le h_K \le \cns'' h \qquad  \forall K\in\Omega_h \ .
\end{equation}
Finally , $\Gammaref_f \,, \ \Gammaref_s\, \ \Gammaref_c \subset
  \Gammaref = \partial \Omegaref$  denote the preimage of $\Gamma_f\,\
  \Gamma_s$ and $\Gamma_c$, respectively.

%---------------------------------------------------------------------------
\section{ Discretization of the scalar and vector fields }\label{sec:disc-spaces}
%---------------------------------------------------------------------------

\subsection{Spline spaces on the parametric domain}
Given the two (horizontal and vertical) knot vectors,  $\mathcal {Z}
_1 = \{\zeta_{1,1},
\dots, \zeta_{m_1,1} \}$ and  $ \mathcal {Z}
_2 = \{\zeta_{1,2}, \dots, \zeta_{m_2,2} \}$, and
the associated  mesh $\Omegaref_h$ on the parametric
domain $\widehat \Omega$, with $1 \leq \alpha \leq p+1$,  we introduce the following spaces:
\begin{subequations}
\begin{align}
\label{eq:W-parametricF}
&\WhrefF =S^{p,p}_{\alpha,\alpha}(\Omegaref_h) \\
\label{eq:Rots-parametricF}
&\RotshrefF=  S^{p-1,p}_{\alpha-1,\alpha} (\Omegaref_h) \times
S^{p,p-1}_{\alpha,\alpha-1}(\Omegaref_h) \\ 
\label{eq:Shears-parametricF}
&\ShearshrefF =\nabla \WhrefF  + \RotshrefF
\end{align}
\end{subequations}
Since by construction $\nabla \WhrefF \subseteq \RotshrefF$, it clearly holds that $\ShearshrefF = \RotshrefF$. 

Essential boundary conditions has to be set directly in the function spaces and this brings us to the definition:
\begin{subequations}
\begin{align}
\label{eq:W-parametric}
&\Whref \equiv \Whref (p,\alpha)=\left \{ v \in
  S^{p,p}_{\alpha,\alpha}(\Omegaref_h) \, : \, v  = 0 \text{  on } \Gammaref_s
  \cup \Gammaref_c \right \}  \\
\label{eq:Rots-parametric}
&\Rotshref \equiv \Rotshref  (p,\alpha)= \left \{ \etab \in S^{p-1,p}_{\alpha-1,\alpha} (\Omegaref_h) \times
S^{p,p-1}_{\alpha,\alpha-1}(\Omegaref_h) \, : \, \etab  =
\boldsymbol{0} \text{  on } \Gammaref_c \text{ and } \etab  \cdot \hat \tt= 0
\text { on  } \Gammaref_{sh} \right \}  \\
\label{eq:Shears-parametric}
&\Shearshref \equiv \Shearshref  (p,\alpha)=\nabla \Whref  + \Rotshref
\end{align}
\end{subequations}
where $\hat \tt$ is a unit tangent vector at $\Gammaref_{sh}$. 
Notice that now $\nabla \Whref \subseteq \Rotshref$ only if $\Gammaref_c=\emptyset$ and $\Gammaref_{ss} = \emptyset$ (see e.g., \cite{BSV}). As a conseguence, $\Shearshref$ needs a characterization which is the object of the next Lemma:

\begin{lemma}\label{lem:XX}
{ The following characterization holds}
  \begin{equation}
    \label{eq:shear-characterization}
\esseb \in \Shearshref  \Leftrightarrow \left \{
  \begin{aligned}
    &\esseb \in \   \RotshrefF\\
 & \esseb  \cdot \tt= 0 \text{ on  } \Gammaref_{sh}   \cup \Gammaref_c
 \\
&\esseb (\widehat {\boldsymbol x}) = \boldsymbol{0} , \text{ for every
  corner } \widehat {\boldsymbol x} \in \partial \Omegaref \text{ such that }
\widehat {\boldsymbol x} \in   \overline{\Gammaref}_{ss}   \cap
\overline{ \Gammaref}_c.
  \end{aligned}
\right .
  \end{equation}
Moreover, given $\esseb \in  \Shearshref $, there are $ v \in \Whref$,
$\etab \in \Rotshref$, such that
\begin{equation}
  \label{eq:2}
  \begin{aligned}
    \etab - \nabla v & =\esseb ,\\ \| \etab\| _{H^1(\Omegaref)} + \| v
    \| _{H^1(\Omegaref)}   &    \leq C h^{-1} \| \esseb\| _{L^2(\Omegaref)} .
  \end{aligned}
\end{equation}
\end{lemma}
\begin{proof}
Denote $\Shearshtilde $ the space of fields that fulfill the
characterization on the right hand side of
(\ref{eq:shear-characterization}).   Let $ \Shearshref  \ni \esseb  = \etab^{\esseb} - \nabla v^{\esseb}$.
 Then,  from (\ref{eq:derivative-of-splines}),  we easily
  see that $ \nabla v^{\esseb}$ belongs to $
  S^{p-1,p}_{\alpha-1,\alpha} (\Omegaref_h) \times
    S^{p,p-1}_{\alpha,\alpha-1}(\Omegaref_h)    $,   and the
    boundary conditions that define
  (\ref{eq:W-parametric})--(\ref{eq:Shears-parametric}) implies that
    $\etab^{\esseb} - \nabla v^{\esseb}$ fulfills the homogeneous
    boundary conditions stated in  (\ref{eq:shear-characterization}). Thus $\Shearshref  \subseteq
    \Shearshtilde $.

In order to prove the inclusion $\Shearshref  \supseteq
    \Shearshtilde $, notice that the fields in $  \Shearshtilde $
    which vanish on $\Gammaref_c$ belong to $\Rotshref$, therefore we
    need to show that the B-spline  basis functions of  $
    \Shearshtilde $ whose support intersect  $\Gammaref_c$ belong to
    $\nabla \Whref$. This is again easily seen from the structure of
    the spaces.

Let again $ \Shearshref  \ni \esseb  = \etab^{\esseb} - \nabla
v^{\esseb}$. Define $v \in \Whref$ such that
\begin{equation}
  \label{eq:splitting-proof-1}
    \frac{\partial v}{\partial \nn} = - \esseb  \cdot \nn \text{ on }
    \Gammaref_c
\end{equation}
and such that  the B-spline coefficients of $v$ not involved in
(\ref{eq:splitting-proof-1}) are set to zero. Because of the locality
of the B-spline basis functions, $ v$ is supported in a  strip of
element associated with the $p+1$ knot spans around  $\Gammaref_c
$. By construction,  the norm of $v$ is bounded by the norm of
(\ref{eq:splitting-proof-1}), then, by a  scaling argument  and  inverse estimate,
\begin{equation}
  \label{eq:splitting-proof-2}
  \begin{aligned}
  \| \nabla v \| _{L^2(\Omegaref)} & \leq C h^{1/2}   \left \|   \frac{\partial
    v}{\partial \nn}  \right \| _{L^2(\Gammaref_c)}\\  & = C h^{1/2}
\left \| \esseb  \cdot \nn  \right \| _{L^2(\Gammaref_c)}  \\  &
\le C \left \| \esseb   \right \| _{L^2(\Omegaref)}.
  \end{aligned}
\end{equation}

Since $\esseb + \nabla v $ vanishes on  $\Gammaref_c$, we can set
\begin{equation*}
%  \label{eq:splitting-proof-3}
  \etab = \esseb + \nabla v \in \Rotshref,
\end{equation*}
and, using again an inverse estimate and (\ref{eq:splitting-proof-2}),  the following estimate holds
\begin{equation}
  \label{eq:splitting-proof-4}
  \begin{aligned}
  \|\etab  \| _{H^1(\Omegaref)} & \leq C h^{-1}    \|\etab  \|
  _{L^2(\Omegaref)}  \\  & = C h^{-1} \left ( \| \esseb  \|
  _{L^2(\Omegaref)}  + \| \nabla v    \|
  _{L^2(\Omegaref)}   \right ) \\  & \leq  C h^{-1} \| \esseb  \|
  _{L^2(\Omegaref)} .
  \end{aligned}
\end{equation}
The property (\ref{eq:2}) follows from (\ref{eq:splitting-proof-2})
and (\ref{eq:splitting-proof-4}).
\end{proof}

\subsection{Spline spaces on the physical domain}\label{sec:splinespacephysdom}
%------------------------------------------------------------------------------------------
Once the finite dimensional spaces $\Whref$, $\Rotshref$ and
$\Shearshref $ on the parametric domain $\Omegaref$ have been defined, we construct the corresponding spaces
$\Wh$, $\Rotsh$ and $\Shearsh$ in the physical domain $\Omega$.

The deflections space $\Wh$ is mapped from the reference domain via the
geometrical parametrization $\bF : \Omegaref \longrightarrow  \Omega  $, that is
\begin{equation*}
%  \label{eq:pressure-mapping}
     \Wh  =  \{ \hat v \circ \bF^{-1}: \hat v  \in \Whref  (p,\alpha )\}.
\end{equation*}
For the rotations and shears  spaces we use the covariant map:

\begin{equation*}
%  \label{eq:piola-kind-mapping-rots}
  \begin{aligned}
     \Rotsh &=\left \{ D\bF^{-T}  \hat \bv  \circ \bF^{-1}:  \hat \bv\in  \Rotshref  \right \},\\
     \Shearsh &=\left \{ D\bF^{-T}  \hat  \bv \circ \bF^{-1} : \hat  \bv \in  \Shearshref   \right \},
 \end{aligned}
\end{equation*}
 where $D\bF^{-T}  = (D\bF^{-1})^T$ is the transpose of the gradient
 of the inverse of the geometrical map $\bF$. 
The push-forward by the covariant map has two main properties: \emph{i)} it preserves the
nullity  of tangential components,  
\emph{ii)} it maps gradient to gradients, i.e., by the chain rule, we have $\nabla w = D\bF^{-T} \widehat{\nabla} \hat w \circ \bF^{-1} $ when $w=\hat w \circ \bF^{-1}$.
Thus, we have that $\bv\cdot\tt =0$ on $\Gamma_{sh} $ for all $\bv\in \Rotsh$ and, recalling Lemma \ref{lem:XX}, the following holds:
\begin{lemma}
  We have
  \begin{equation*}
%    \label{eq:Shears-physical-domain}
    \Shearsh = \nabla  \Wh  +   \Rotsh,
  \end{equation*}
and, for  $\esseb \in  \Shearsh $, there exist $ v \in \Wh$,
$\etab \in \Rotsh$, such that
\begin{equation}
  \label{eq:Shears-physical-domain-decomposition}
  \begin{aligned}
    \etab - \nabla v & =\esseb ,\\ \| \etab\| _{H^1(\Omega)} + \| v
    \| _{H^1(\Omega)}   &    \leq C h^{-1} \| \esseb\| _{L^2(\Omega)} .
  \end{aligned}
\end{equation}
\end{lemma}

  % \begin{proof}
%     We need to show that the space of gradients $\nabla \Whref$ on the reference domain $\Omegaref$ is
%     transformed     into  $\nabla \Wh$ on the physical domain $\Omega$, by
%     (\ref{eq:piola-kind-mapping-rots}). This is just  a
%     straightforward application of the chain rule: starting from
%     (\ref{eq:pressure-mapping}) we have
%     \begin{equation}\label{eq:Shears-physical-domain-proof}
%       v = \hat v \circ \bF^{-1},
%     \end{equation}
% with $v \in \Vh$,  $\hat v \in \Vhref$; deriving
% (\ref{eq:Shears-physical-domain-proof}) we immediately obtain
% \begin{displaymath}
%   \nabla v = D\bF^{-T}   ( \nabla \hat v )\circ \bF^{-1}.
% \end{displaymath}
% Then, the decomposition
% (\ref{eq:Shears-physical-domain-decomposition}) follows from
% (\ref{eq:2}) and change of variable.
%   \end{proof}

% The next lemma summarizes the approximation properties of the
% spaces $\Wh$, $\Rotsh$ and $\Shearsh$.

We now turn to approximation properties and we make use of the approximation results proved in \cite{BBCHS06} and adapted in \cite{BSV} for vector fields under covariant transformation. 

\begin{lemma}\label{lemma:approximation-properties}
Let $1 \le  \alpha \leq p+1 $, and let $\bF \in (\cC^\infty
_{\alpha,\alpha} (\Omegaref_{h}))^2$.   There exist projectors $\Pi_{\Wh}: \W_{0} \rightarrow  \Wh$,
  $\Pi_{\Rotsh}: \Rots_{0} \rightarrow  \Rotsh $ and
  $\Pi_{\Shearsh}: \Shears \rightarrow  \Shearsh $ such that,
  for all $1 < s \leq p+1$
  \begin{equation*}
%    \label{eq:error-estimate-W}
    \begin{aligned}
      \| v-\Pi_{\Wh}(v)\|_{ H^1(K) }  &\leq C h^{s-1} | v
      |_{H^{s}(\tilde K) }, \qquad \forall K \in
       {\Omega}_h,\forall  v \in \W_{0} \cap
      H^{s}(\tilde K) ,
    \end{aligned}
  \end{equation*}
for all $1 \le s \leq p$
  \begin{equation*}
%    \label{eq:error-estimate-Rots}
    \begin{aligned}
    \| \etab-\Pi_{\Rotsh}(\etab)\|_{ H^1(K) }  &\leq C h^{s-1} |
    \etab |_{H^{s}(\tilde K) }, \qquad \forall K \in
       {\Omega}_h,\forall \etab \in \Rots_{0} \cap
    (H^{s}(\tilde K))^2
    \end{aligned}
  \end{equation*}
for all $1 < s \leq p$
  \begin{equation*}
%    \label{eq:error-estimate-Shears}
    \begin{aligned}
   \| \esseb-\Pi_{\Shearsh}(\esseb)\|_{ L^2(K) }  &\leq C
   h^{s} | \esseb |_{H^{s}(\tilde K)}, \qquad \forall K \in
       {\Omega}_h, \forall \esseb   \in  \Shears^s,
    \end{aligned}
  \end{equation*}
where
$$
\Shears^s = \left (\Rots_{0} +  \nabla \W _{0}\right ) \cap  (H^{s}(\tilde
K))^2. 
$$

\end{lemma}

\begin{proof}
The projector $\Pi_{\Wh}$ can be chosen as the quasi-interpolant proposed in  \cite{BBCHS06} and  homogeneuous boundary condition are imposed by setting to zero che coefficient of all the basis functions which are non-zero on $\Gamma_c$ and $\Gamma_s$. 
The projector $\Pi_{\Rotsh}$ is the one constructed in \cite{BSV}, with boundary condition set only on $\Gamma_c$ and $\Gamma_{sh}$;  and $\Pi_{\Shearsh}$ is the same with the set of boundary condition corresponding to \eqref{eq:shear-characterization}. 
\end{proof}

\begin{rem}\label{ref:diag:phis} If $\: \Gamma_{c}
=\emptyset$ and $\Gamma_{ss} = \emptyset$, then $\nabla \Wh \subseteq \Rotsh =
\Shearsh $. In this case, we can choose as interpolation operators the commuting projectors 
that are constructed in \cite{BRSV}: i.e., there are projectors $P_{\Wh}$ and $P_{\Rotsh}$ such that 
  \begin{equation*}
%    \label{eq:commuting-property-phis}
    P_{\Rotsh} \nabla  = \nabla  P_{\Wh}.
  \end{equation*}
We refer to \cite{BRSV, BSV} for the details.
\end{rem}

%---------------------------------------------------------------------------
\subsection{The discrete problem} 
%---------------------------------------------------------------------------

We can now introduce our proposed method. The Galerkin formulation for the approximation of Problem~\eqref{P}, reads
\begin{equation}\label{Ph}
\left\{
\begin{aligned}
& \text{Find  $(\rots_h,\w_h) \in \Rots_{h,0}\times \W_{h,0}$  such  that} \\
& a(\rots_h,\etab_h)+ t^{-2} (\rots_h - \nabla \w_h,
\etab_h-\nabla v_h) = (f,v_h) \quad \forall (\etab_h,v_h) \in \Rots_{h,0}\times \W_{h,0} ;
\end{aligned}
\right.
\end{equation}
this is the discrete problem which is solved in our isogeometric
code. For the purpose of its stability and  convergence analysis, we introduce the discrete shear stress
$$
\gammab_h = t^{-2} (\rots_h - \nabla \w_h) \in \Shears_{h,0} ,
$$
and reformulate Problem \eqref{Ph} in the following mixed form
(analogous to Problem~\eqref{Pmixed}):

\begin{equation}\label{Pmixedh}
\left\{
\begin{aligned}
& \text{Find  $\rots_h \w_h,\gammab_h)\in \Rots_{h,0} \times \W_{h,0}\times \Shears_{h,0}$ such  that} \\
& a(\rots_h,\etab_h)+ (\shears_h, \etab_h-\nabla v_h) = (f,v_h)
\quad \forall
\etab_h \in \Rots_{h,0}, v_h \in \W_{h,0} \\
&(\rots_h - \nabla \w_h,\esseb_h)- t^2 (\shears_h,\esseb_h)=0
\quad \forall \esseb_h \in \Shears_{h,0} \ .
\end{aligned}
\right.
\end{equation}

Introducing, for all $(\betab,u,\taub)$ and $(\etab,v,\esseb)$ in
$\Rots_{0}\times\W_{0}\times\Shears$, the symmetric bilinear form
\begin{equation}\label{bil}
\bil(\betab,u,\taub;\etab,v,\esseb) = a(\betab,\etab) + (\taub,
\etab-\nabla v) + (\betab - \nabla u,\esseb) - t^2 (\taub,\esseb)
,
\end{equation}
the continuous and discrete problems can be also written as
\begin{equation}\label{Pb}
\bil(\rots,\w,\shears;\etab,v,\esseb) = (f,v) \quad \forall
(\etab,v,\esseb) \in \Rots_{0}\times\W_{0}\times\Shears\ ,
\end{equation}
and
\begin{equation}\label{Pbh}
\bil(\rots_h,\w_h,\shears_h;\etab_h,v_h,\esseb_h) = (f,v_h) \quad
\forall (\etab_h,v_h,\esseb_h) \in
\Rots_{h,0}\times\W_{h,0}\times\Shears_{h,0} \ .
\end{equation}

Note that, differently from what happens with most finite element methods (e.g. the well-known MITC elements, cf.~\cite{brebaf, BFS91}), if one
considers the limit case $t=0$ in formulation \eqref{Pmixedh}, the
second equation gives exactly $\rots_h = \nabla \w_h$. Therefore,
substituting such identity in the first equation of
\eqref{Pmixedh}, one gets the problem
\begin{equation*}%\label{Kirchhoff}
\left\{
\begin{aligned}
& \text{Find  $\w_h \in
\W_{h,0}$  such  that} \\
& a(\nabla \w_h,\nabla v_h) = (f,v_h) \quad \forall v_h \in \W_{h,0} \
,
\end{aligned}
\right.
\end{equation*}
i.e. the Kirchhoff plate bending problem.

%---------------------------------------------------------------------------
\section{Convergence analysis}\label{sec:conv}
%---------------------------------------------------------------------------

In this section we investigate the convergence properties of the
method proposed in Section~\ref{sec:disc-spaces}.  

Under the assumption  $\Gamma_c=\Gamma_{ss}=0$, we present in
Section~\ref{sec:simple-conv-analys} an analysis based only upon
coercivity. A similar approach is proposed by  Duran and Libermann in
\cite{DL92}, but things are straightforard in our case, thanks to the
commuting projectors that are at our disposal (see Remark
\ref{ref:diag:phis}).

On the other hand, for general boundary conditions, 
we need a more delicate analysis which is the object of  Section \ref{sec:discr-norm}. 

\subsection{Coercivity-based convergence analysis}
\label{sec:simple-conv-analys}

We assume in this section:
\begin{assum}
\label{ass:2}  It holds
$\Gamma_{c} =\emptyset$ and $\Gamma_{ss} = \emptyset$.
\end{assum}

We recall that, under the above assumption, as noted in
Remark \ref{ref:diag:phis} the following properties hold
\begin{equation*}
%    \label{eq:commuting-property-phis-2}
    \Rotsh = \Shearsh \ , \quad
    P_{\Rotsh} \nabla  = \nabla  P_{\Wh}.
\end{equation*}

The following lemma holds.

\begin{lemma}\label{dl-lemma}
Let $(\rots,\w)$ be the solution of Problem \eqref{P} and
$(\rots_h,\w_h)$ be the solution of Problem \eqref{Ph}. Let the
shears $\shears=t^{-2}(\rots-\nabla\w)$ and
$\shears_h=t^{-2}(\rots_h-\nabla\w_h)$.  
Then for $1 < s \le p$ it holds 
\begin{equation}\label{dl-errest}
|| \rots-\rots_h ||_{H^1(\Omega)} +  ||w-w_h  ||_{H^1(\Omega)} + t
||\shears-\shears_h ||_{L^2(\Omega)}
 \le C h^{s-1} (  || \rots ||_{H^{s} (\Omega)} +  t ||\shears ||_{H^{s-1}(\Omega)} ) \ ,
\end{equation}
where the constant $C$ is independent of $h,t$.
\end{lemma}

{\it Proof.}  From~\eqref{P} and~\eqref{Ph}, we obtain

\begin{equation*}%\label{dl-1}
a(\rots-\rots_h,\etab_h)+ (\shears-\shears_h, \etab_h-\nabla v_h)
= 0 \quad \forall \etab_h \in \Rots_h, v_h \in \W_h .
\end{equation*}

Therefore, for every $(\rots_I,\w_I)\in\Rots_{0,h}\times\W_{0,h}$, and
$\shears_I := t^{-2}(\rots_I-\nabla\w_I)\in\Shears_{0,h}$, we infer

\begin{equation}\label{dl-2}
\begin{aligned}
a(\rots_I-\rots_h,\etab_h)  + (\shears_I  &  -\shears_h ,
\etab_h-\nabla v_h) =  a(\rots_I-\rots,\etab_h)\\
& + (\shears_I-\shears, \etab_h-\nabla v_h) \qquad \forall \etab_h
\in \Rots_h, v_h \in \W_h .
\end{aligned}
\end{equation}

Choosing $\etab_h = \rots_I-\rots_h$ and $v_h =w_I- w_h$, we get

$$
t^2(\shears_I-\shears_h )= \etab_h-\nabla v_h.
$$

Therefore, from~\eqref{dl-2} we get

\begin{equation}\label{dl-3}
\begin{aligned}
a(\rots_I-\rots_h,\rots_I-\rots_h)  + t^2 (\shears_I   -\shears_h
,
\shears_I -\shears_h)  =  a(& \rots_I-\rots,\rots_I-\rots_h)\\
& + t^2 (\shears_I-\shears, \shears_I-\shears_h) .
\end{aligned}
\end{equation}
Using standard arguments, from~\eqref{dl-3} we obtain

\begin{equation*}%\label{dl-4}
|| \rots_I-\rots_h ||_{H^1(\Omega)} +  t ||\shears_I-\shears_h
||_{L^2(\Omega)}
 \le C (  || \rots_I-\rots ||_{H^1(\Omega)} +  t ||\shears_I-\shears ||_{L^2(\Omega)}) \ .
\end{equation*}
 We choose  now $\rots_I=P_{\Rotsh} \rots$ and
$w_I=P_{\Wh} w$, where $P_{\Rotsh}$ and $P_{\Wh} w$ are the
interpolation operators introduced in Section
\ref{sec:splinespacephysdom}. By the choice of $\shears_I$, we obtain:

\begin{equation*}%\label{dl-5}
\shears_I := t^{-2}(\rots_I-\nabla\w_I) = t^{-2}P_{\Rotsh}
(\rots - \nabla\w) = P_{\Rotsh} \shears .
\end{equation*}

Therefore, also using Lemma \ref{lemma:approximation-properties},
we have

\begin{equation}\label{dl-6}
 t ||\shears_I-\shears ||_{L^2(\Omega)} = t ||P_{\Rotsh}\shears-\shears ||_{L^2(\Omega)} \le C h^{s-1} t || \shears||_{H^{s-1}(\Omega)} .
\end{equation}

Furthermore, again due to Lemma
\ref{lemma:approximation-properties}, it holds

\begin{equation}\label{dl-7}
||\rots_I - \rots||_{H^1(\Omega)} = ||   P_{\Rotsh}\rots - \rots
||_{H^1(\Omega)}\le C h^{s-1} || \rots||_{H^{s-1}(\Omega)} .
\end{equation}

Using the triangle inequality and~\eqref{dl-6}-\eqref{dl-7}, from
Lemma~\ref{dl-lemma} we get

\begin{equation}\label{dl-8}
|| \rots-\rots_h ||_{H^1(\Omega)} +  t ||\shears-\shears_h
||_{L^2(\Omega)}
 \le C h^{s-1} (  || \rots ||_{H^{s}(\Omega)}  +  t ||\shears ||_{H^{s-1}(\Omega)} ) \ .
\end{equation}

To get the error estimate on the deflections, we simply notice
that it holds

\begin{equation}\label{dl-9}
\nabla(w-w_h) =(\rots-\rots_h) - t^2(\shears-\shears_h) .
\end{equation}

Estimate~\eqref{dl-errest} now easily follows
from~\eqref{dl-8}-\eqref{dl-9}.

\fin

% \begin{rem}
% We notice that, contrary to the analysis developed in
% Section~\ref{sec:conv}, the present error analysis does not assume
% the quasi-uniformity condition on the meshes. However, error
% estimate~\eqref{dl-errest} is weaker than the one shown in
% Corollary~\ref{prop:conv-rate}. In fact, in~\eqref{dl-errest} we
% control only $t || \shears - \shears_h||_{L^2(\Omega)}$, while
% Corollary~\ref{prop:conv-rate} provides an estimate also on $h ||
% \shears - \shears_h||_{L^2(\Omega)}$ (which may be considered as a
% sort of \emph{discrete} $H^{-1}$ norm).

% Finally, we remark that the arguments of
% Section~\ref{sec:improved} apply without any quasi-uniformity
% assumptions. As a consequence, the improved error estimate on the
% deflections  (cf. Lemma~\ref{lemma:deflH1}) still holds.
% \end{rem}

%---------------------------------------------------------------------------
\subsection{Convergence in a discrete norm} 
\label{sec:discr-norm}
%---------------------------------------------------------------------------

In this section we suppose that general boundary conditions are set and thus 
Assumption \ref{ass:2} does not hold. On the other hand, we remind that \eqref{quasiunif-phys} holds true and this is 
crucial in the subsequent analysis. 

We will make use of the following discrete norm~\cite{chasten}: 
\begin{equation}\label{norms}
\begin{aligned}
& ||| \etab, v |||_{\Rots}^2 = || \etab ||_{H^1(\Omega)}^2 +
\sum_{K\in\Omega_h}
\frac{1}{t^2+h_K^2} || \nabla v - \etab ||_{L^2(K)}^2 \\
& ||| \esseb |||_\Shears^2 = t^2 || \esseb ||_{L^2(\Omega)}^2 +
\sum_{K\in\Omega_h} h_K^2 || \esseb ||_{L^2(K)}^2 \\
& ||| \etab, v, \esseb |||^2 = ||| \etab, v |||_{\Rots}^2 + |||
\esseb |||_\Shears^2 \ ,
\end{aligned}
\end{equation}
for all $\etab\in\Rots_{0}$, $v\in\W_{0}$ and $\esseb\in \Shears$.
It is easy to check that it holds
$$
||| \etab, v |||_{\Rots}^2 \ge C || v ||_{H^1(\Omega)}^2 \quad
\forall \etab\in\Rots_{0}, \ v\in\W_{0} \ ,
$$
with $C$ independent of $h$ and $t$. Therefore, the norm $|||\cdot|||_{\Rots}$ also
includes the $H^1$ norm of both rotations and deflections.

% ------------------
% In the following, we assume that the mesh $\widehat{\Omega}_h$
% is shape regular, i.e. the ratio between the largest and the
% shortest edge of all elements $Q \in\widehat{\Omega}_h$ is
% uniformly bounded. We will present in Section \ref{sec:DLstyle} an alternative convergence analysis,
% involving different norms, that does not make use of the shape regularity condition.
% Notice that shape regularity  implies a quasi-uniformity condition
% in this context,  since $\widehat{\Omega}_h$ is a tensor product
% mesh. Due to the fact that the map $\F$ is fixed at the coarsest
% level of discretization and piecewise regular, the
% quasi-uniformity condition extends also to the physical mesh
% $\Omega_h$. We formalize such remark in the following.

% \begin{assum}\label{ass1}
% The parametric mesh $\widehat{\Omega}_h$ is shape regular. As a
% consequence, there exist positive constants $\cns$, $\cns'$,
% $\cns''$, fixed once and for all, such that
% \begin{equation}\label{quasiunif}
% \cns h \le h_Q \le h \ , \quad \cns' h \le h_K \le \cns'' h \qquad
% \forall Q \in \widehat{\Omega}_h, \ \forall K\in\Omega_h \ .
% \end{equation}
% \end{assum}

The following stability result can be shown adopting well-known
techniques, cf.~\cite{chasten}, combined with the results Section~\ref{sec:disc-spaces}. 
Therefore, the proof will be only sketched and several details will be omitted.

\begin{prop}\label{prop:stab}
For all $(\betab_h,u_h,\taub_h)\in\Rots_{0,h},\W_{0,h},\Shears_{0,h}$, there
exists $(\etab_h,v_h,\esseb_h))\in\Rots_{0,h},\W_{0,h},\Shears_{0,h}$ such that
\begin{eqnarray}
&& \bil(\betab_h,u_h,\taub_h;\etab_h,v_h,\esseb_h) \ge C_1
||| \betab_h, u_h, \taub_h |||^2 \ , \label{stab:a}\\
&& ||| \etab_h, v_h, \esseb_h ||| \le C_2 ||| \betab_h, u_h, \taub_h
||| \ , \label{stab:b}
\end{eqnarray}
with $C_1$ and $C_2$ positive constants independent of $h$ and $t$.
\end{prop}
{\it Proof.} Taking $(\etab_h^1,v_h^1,\esseb_h^1)=(\betab_h,u_h,-\taub_h)$ in~\eqref{bil}, one immediately gets
\begin{eqnarray}
&& \bil(\betab_h,u_h,\taub_h;\etab_h^1,v_h^1,\esseb_h^1) \ge  
C ||\betab_h ||_{H^1(\Omega)}^2 + t^2 || \taub_h ||_{L^2(\Omega)}^2 \ , \label{eq6} \\
&& ||| \etab_h^1, v_h^1, \esseb_h^1 ||| \le ||| \betab_h, u_h,
\taub_h ||| \ . \label{eq7}
\end{eqnarray}
Due to (\ref{eq:Shears-physical-domain-decomposition}), we can choose
$(\etab_h^2,v_h^2,\esseb_h^2)\in \Rots_{0,h}\times\W_{0,h} \times\Shears_{h,0}$ such that
\begin{equation}
  \label{eq:temporanea}
  \begin{aligned}
 & \esseb_h^2 = \boldsymbol{0}\\
   & \etab_h^2 - \nabla v_h^2 = h^2
\taub_h \\ &\|  \etab_h^2 \| _{H^1(\Omega)} + \|  v_h^2   \|
_{H^1(\Omega)} \leq C h  \| \taub_h   \| _{L^2(\Omega)} \ .
  \end{aligned}
\end{equation}
The Cauchy-Schwarz inequality yields
\begin{equation}\label{eq1}
\begin{aligned}
\bil(\betab_h,u_h,\taub_h;\etab_h^2,v_h^2,\esseb_h^2) & =
a(\betab_h,\etab_h^2) + h^2 || \taub_h ||_{L^2(\Omega)}^2 \\
& \ge h^2 || \taub_h ||_{L^2(\Omega)}^2 - C  || \betab_h
||_{H^1(\Omega)} || \etab_h^2 ||_{H^1(\Omega)} \\
& \ge h^2 || \taub_h ||_{L^2(\Omega)}^2 - C h || \betab_h
||_{H^1(\Omega)} || \taub_h^2 ||_{H^1(\Omega)} \ .
\end{aligned}
\end{equation}
Applying the arithmetic-geometric mean inequality, from (\ref{eq:temporanea})--\eqref{eq1} one easily gets 
\begin{equation}\label{eq2}
\bil(\betab_h,u_h,\taub_h;\etab_h^2,v_h^2,\esseb_h^2) \ge
\frac{1}{2} h^2 || \taub_h ||_{L^2(\Omega)}^2 - C' || \betab_h
||_{H^1(\Omega)}^2 \ .
\end{equation}
Furthermore,  by definition \eqref{norms},  recalling
\eqref{quasiunif} and (\ref{eq:temporanea}),  one has
\begin{equation}\label{eq3}
||| \etab_h^2, v_h^2, \esseb_h^2 ||| \le C \Big( h|| \taub
||_{L^2(\Omega)} + \frac{h^2}{t+h} || \taub ||_{L^2(\Omega)} \Big)
\le C ||| \betab_h, u_h, \taub_h ||| \ .
\end{equation}
We finally select $(\etab_h^3,v_h^3,\esseb_h^3)=(0 ,0, (t+h)^{-2} (\betab_h-\nabla u_h) )$. Using again the
arithmetic-geometric mean inequality and some simple algebra,
we obtain
\begin{equation}\label{eq4}
\begin{aligned}
\bil(\betab_h,u_h,\taub_h;\etab_h^3,v_h^3,\esseb_h^3) & =
\frac{1}{(t+h)^2} || \betab_h - \nabla u_h ||_{L^2(\Omega)}^2 -
\frac{t^2}{(t+h)^2} (\taub,\betab_h-\nabla u_h) \\
& \ge \frac{1}{2(t+h)^2} || \betab_h - \nabla u_h
||_{L^2(\Omega)}^2 - C
\frac{t^4}{(t+h)^2} || \taub_h ||_{L^2(\Omega)}^2 \\
& \ge \frac{1}{2(t^2+h^2)} || \betab_h - \nabla u_h
||_{L^2(\Omega)}^2 - C t^2 || \taub_h ||_{L^2(\Omega)}^2 \ .
\end{aligned}
\end{equation}
Moreover, recalling \eqref{norms} and \eqref{quasiunif}, it
follows
\begin{equation}\label{eq5}
||| \etab_h^3, v_h^3, \esseb_h^3 ||| \le C (t+h)^{-1} ||
\betab_h-\nabla u_h ||_{L^2(\Omega)} \le C ||| \betab_h, u_h,
\taub_h ||| \ .
\end{equation}
We now consider the linear combination
$$
(\etab_h,v_h,\esseb_h) = \sum_{i=1}^3  c_i
(\etab_h^i,v_h^i,\esseb_h^i) \ ,
$$
with positive $c_i\in\R$. 
Estimates~\eqref{stab:a} and~ \eqref{stab:b} follow from a suitable choice of the coefficients $c_i$, by using  
\eqref{eq6}--\eqref{eq7}, \eqref{eq2}--\eqref{eq3} and \eqref{eq4}--\eqref{eq5}.

\fin

From Proposition \ref{prop:stab}, the following result is easily deduced.

\begin{prop}\label{prop:conv-interp}
Let $(\rots,\w)$ be the solution of Problem \eqref{P} and
$(\rots_h,\w_h)$ be the solution of Problem \eqref{Ph}. Let the
shears $\shears=t^{-2}(\rots-\nabla\w)$ and
$\shears_h=t^{-2}(\rots_h-\nabla\w_h)$. Then, for any
$(\rots_I,\w_I,\shears_I)$ in $\Rots_{0,h}\times\W_{0,h}\times\Shears_{0,h}$
it holds
$$
||| \rots-\rots_h, \w-\w_h, \shears-\shears_h ||| \le C |||
\rots-\rots_I,\w-\w_I,\shears-\shears_I ||| \ ,
$$
where the constant $C$ is independent of $h,t$.
\end{prop}
{\it Proof.} The proof  is a consequence of standard consistency-stability arguments. We apply Proposition
\ref{prop:stab} to the difference
$(\rots_h-\rots_I,\w_h-\w_I,\shears_h-\shears_I)$, then apply the
error equation given by~\eqref{Pb} and \eqref{Pbh}, to derive
\begin{equation}\label{eq10}
||| \rots_h-\rots_I,\w_h-\w_I,\shears_h-\shears_I |||^2 \le C
\bil(\rots-\rots_I,\w-\w_I,\shears-\shears_I;\etab_h, v_h,
\esseb_h) \ ,
\end{equation}
where
\begin{equation}\label{eq11}
||| \etab_h, v_h, \esseb_h ||| \le |||
\rots_h-\rots_I,\w_h-\w_I,\shears_h-\shears_I ||| \ .
\end{equation}
Observing that the bilinear form $\bil$ is bounded with respect to
the norm $||| \cdot |||$ and applying \eqref{eq11}, bound~\eqref{eq10} gives
$$
||| \rots_h-\rots_I,\w_h-\w_I,\shears_h-\shears_I ||| \le C |||
\rots-\rots_I,\w-\w_I,\shears-\shears_I ||| \ .
$$
The result follows using the triangle inequality.

\fin

Combining the above proposition with the interpolation results of
Lemma~\ref{lemma:approximation-properties} the next Corollary easily follows.
\begin{corol}\label{prop:conv-rate}
Let $(\rots,\w)$ be the solution of Problem \eqref{P} and
$(\rots_h,\w_h)$ be the solution of Problem \eqref{Ph}. Let the
shears $\shears=t^{-2}(\rots-\nabla\w)$ and
$\shears_h=t^{-2}(\rots_h-\nabla\w_h)$. Then, provided the
solution of \eqref{P} is sufficiently regular for the right-hand-side to make sense, for all $1 < s \le p$ it holds

\begin{equation*}%\label{eq:err-est-dnorm}
||| \rots-\rots_h, \w-\w_h, \shears-\shears_h ||| \le C h^{s-1}
\Big( || \rots ||_{s} + || \w ||_{s+1} + || \shears ||_{{\rm
max}(0,s-2)} + t || \shears ||_{s-1} \Big) \ ,
\end{equation*}
where the constant $C$ is independent of $h$ and $t$.
\end{corol}

We note that the regularity required in Corollary
\ref{prop:conv-rate} may be unrealistic for high values $s$, due to the presence of boundary layers
(see for instance \cite{af2}). 
%
%On the other hand, if the loading is regular and
%any internal domain $\Omega_{int}$ with positive distance from
%$\Gamma$ is considered, then the regularity assumed above does
%hold on $\Omega_{int}$. Thus, in the presence of boundary layers,
%the discrete solution is expected to converge with the higher rate
%$O(h^s)$ inside the domain and with a lower rate near the
%boundary.

%--------------------------------------------------------------------------
\subsection{Improved convergence rates}\label{sec:improved}
%--------------------------------------------------------------------------

We start this section with the following lemma, which shows
an improved convergence property for the $L^2$ norm of the rotations. 
We omit the proof,  since it merely adopts a classical Aubin-Nietsche argument.

\begin{lemma}\label{lemma:rotsL2}
Let the Problem~\eqref{P} be regular. Then, under the notation of Proposition~\ref{prop:conv-interp}, it holds
$$
|| \rots - \rots_h ||_{L^2(\Omega)} \le C h \Big( || \rots-\rots_h
||_{H^1(\Omega)} + t || \shears - \shears_h ||_{L^2(\Omega)} \Big)
$$
with $C$ independent of $h,t$.
\end{lemma}

We remark that the regularity of Problem~\eqref{P} holds, for
instance, when the domain $\Omega$ is convex and the plate is clamped on the whole boundary \cite{RS11}.
For more general cases we refer the reader to \cite{af2,RS11} where the (smooth) boundary layers and the effect of corners are respectively investigated. 

The following result gives an improved convergence rate of convergence for
the deflection variable.

\begin{lemma}\label{lemma:deflH1}
Under the assumptions and notation of Lemma~\ref{lemma:rotsL2},
for any interpolant $\w_I\in\W_{h,0}$, it holds
$$
|| \w - \w_h ||_{H^1(\Omega)} \le C h \Big( || \rots-\rots_h
||_{H^1(\Omega)} + t || \shears - \shears_h ||_{L^2(\Omega)} \Big) +
|| \nabla (\w-\w_I) ||_{L^2(\Omega)}
$$
with $C$ independent of $h$ and $t$.
\end{lemma}
{\it Proof.} Testing Problems~\eqref{P} and~\eqref{Ph}
with the choice $({\bf 0},v_h)\in\Rots_{h,0} \times \W_{h,0}$, and taking the difference, one obtains
\begin{equation*}%\label{eq8}
(\nabla(\w-\w_h),\nabla v_h) = (\rots-\rots_h,\nabla v_h) \ .
\end{equation*}
Thus, setting $v_h=\w_h-\w_I$ one gets
\begin{equation}\label{eq9}
\begin{aligned}
|| \nabla (\w_h-\w_I) ||_{L^2(\Omega)}^2 & = (\nabla (\w_h-\w_I),\nabla v_h) \\
&= (\nabla (\w_h-\w),\nabla v_h) + (\nabla (\w-\w_I),\nabla v_h)
\\ & = (\rots-\rots_h,\nabla v_h) + (\nabla (\w-\w_I),\nabla v_h) .
\end{aligned}
\end{equation}
Applying the Poincar\'e inequality, then using \eqref{eq9} and the
Cauchy-Schwarz inequality, we obtain
$$
|| \w_h-\w_I ||_{H^1(\Omega)} \le C || \nabla (\w_h-\w_I)
||_{L^2(\Omega)} \le C || \rots-\rots_h ||_{L^2(\Omega)} + ||
\nabla (\w-\w_I) ||_{L^2(\Omega)} \ .
$$
The result trivially follows from the above bound, combined with the
triangle inequality and Lemma \ref{lemma:rotsL2}.

\fin

Note that in the proposed method the approximation order of the
space $\W_h$ is one point higher with respect to the rotation
space $\Rots_h$, as shown in Lemma
\ref{lemma:approximation-properties}. Therefore, the above result
indeed implies an improved rate of convergence with respect to
Proposition \ref{prop:conv-interp}.

%---------------------------------------------------------------------------
%\subsection{An alternative analysis}\label{sec:DLstyle}
%---------------------------------------------------------------------------

%\newpage
%------------------------------------------------------------------------------------------
\section{Numerical tests}\label{sec:nums}
%------------------------------------------------------------------------------------------

In this section we present some numerical experiments to show the
actual performance of the numerical methods introduced in Section~\ref{sec:disc-spaces}. For all tests, we
select a material with Poisson's ratio $\nu = 0.3$ and Young's modulus $E = 1.092 \cdot 10^7 N/m^2$. 
Accordingly, we will always express forces and lengths in $N$ and $m$, respectively.

In all the tests we start from the mesh used for the geometry representation and we build the different components of the approximation spaces performing three simple steps:
\begin{enumerate}
	\item  degree elevation of the basis function used for the mapping, to the requested degree of the given field component;
	\item  knot repetitions  in order to have the proper continuity; 
	\item the $h$-refinement is performed using uniform knot insertions in the original knots vectors.
\end{enumerate}

We remark that, if the geometry is described by NURBS, step 1. is perfomed on the B-splines basis obtained by setting the NURBS weight to one.

\subsection{Case 1: fully hardly clamped square plate.}

We consider  a problem having a known the analytical solution  (see~\cite{lovachinosi}). 
This test consists of a unitary square block $[0,1]^2$ with all the four sides clamped, i.e., $\Gamma= \Gamma_c$, subject to a body load given by
\begin{align*}
 f(x,y) 
 & = \frac{E}{12(1-\nu^2)} \bigl[ 
 12 y (y-1)(5x^2-5x+1) 
 ( 2 y^2 (y-1)^2 \nonumber \\
 &\phantom{=} + x (x-1) (5y^2-5y+1) ) \nonumber \\
 &\phantom{=}
 +12 x (x-1)(5y^2-5y+1) 
 ( 2 x^2 (x-1)^2 \nonumber \\
 &\phantom{=} + y (y-1) (5x^2-5x+1) )
 \bigr]\ .
\end{align*}
The analytical solution is
\begin{align*}
\rots(x,y) 
& =
\begin{pmatrix}
y^3 (y-1)^3 x^2 (x-1)^2 (2x-1) \\ \\
x^3 (x-1)^3 y^2 (y-1)^2 (2y-1)
\end{pmatrix}
\\ 
\nonumber \\
\w(x,y) 
&= \frac{1}{3} x^3 (x-1)^3 y^3 (y-1)^3 \nonumber \\
&\phantom{=} - \frac{2 t^2}{5(1-\nu)} \bigl[
y^3 (y-1)^3 x (x-1) (5x^2-5x+1) \nonumber \\
&\phantom{=} + x^3 (x-1)^3 y (y-1) (5y^2-5y+1)
 \bigr] \; .
\end{align*}

In Figure~\ref{fig:case1_results} we plot the $H^1$-norm approximation error for the displacement field $\w(x,y)$ and for the rotation field $\rots(x,y)$. We run tests with $p=3,\ 4$ ad $\alpha=2,\ 3$ both for deflections and rotations,  respectively, 
and the  thickness $t=10^{-3}$. Figure~\ref{fig:case1_results} clearly displays convergence rates in perfect agreement with the theoretical results of 
Section~\ref{sec:conv}.

In addition, we performed simulations (not reported here) using also $t\in\{ 10^{-1}, 10^{-2}, 10^{-4} \}$, and 
we found that the methods are substantially insensitive to thickness changes.  

\begin{figure}[!h]
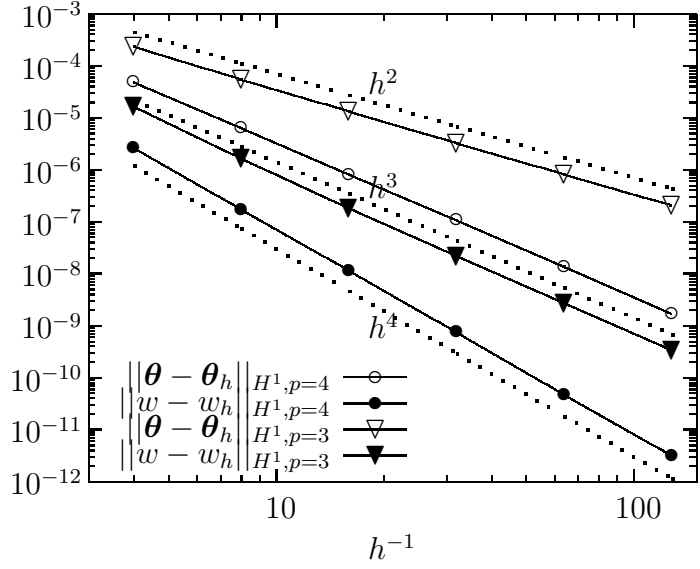

\centering
\include{ReissnerMindlin_Case1_t=1e-3}
\caption{Case 1 (square with four sides hardly clamped),
  $t=10^{-3}$.} \label{fig:case1_results}
\end{figure}

\subsection{Case 2: quarter of an annulus with four hardly simply supported sides.}

The second test consists in a fully hardly simply supported quarter of an annulus plate., i.e., $\Gamma= \Gamma_{sh}$ (see Figure~\ref{fig:case2_geometry}) and loaded by  
$$f\displaystyle{(x,y)=10^4 \; \sin \bigl( 2\arctan(\tfrac{y}{x})} \bigr)\ .$$

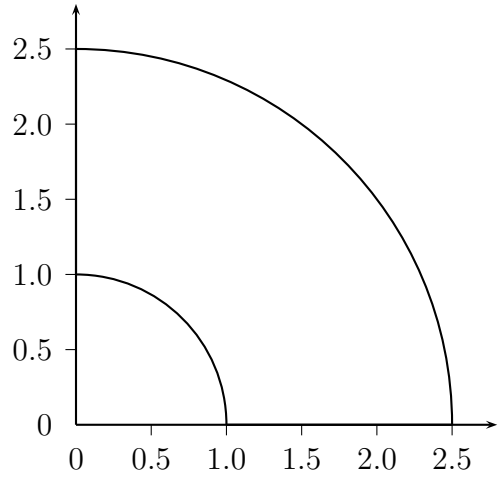
\begin{figure}[!h]
	\centering
	\psset{unit=2cm}
	\begin{pspicture}(0.0,0.0)(3.0,3.0)
		\psaxes[Dx=0.5,Dy=0.5,tickstyle=bottom]{->}(0.0,0.0)(2.8,2.8)
		\psarc(0.0,0.0){1.0}{0}{90}
		\psarc(0.0,0.0){2.5}{0}{90}
		\psline(1.0,0.0)(2.5,0.0)
		\psline(0.0,1.0)(0.0,2.5)
	\end{pspicture}
	\vskip0.5cm
	\caption{Fully hardly simply supported annular plate.} \label{fig:case2_geometry}
\end{figure}

In this case the analytical solution is not available. Therefore, we use as reference solution the one computed on a fine mesh  ($h=1/256$) and with degree $p=3$ and regularity index $\alpha=2$ both for deflections and rotations. The same degree and regularity index are used for the simulations on the coarser meshes.

In Figure~\ref{fig:case2_results_1e-2} and Figure~\ref{fig:case2_results_1e-3} we plot the $H^1$-norm approximation errors for $t=10^{-2}$ and $t=10^{-3}$ respectively. Once again, an excellent agreement with the theoretical predictions can be noticed.
% with the theory for the space we used, i.e. an approximation error $O(h^3)$ for the deflection field and $O(h^2)$ for the rotations field. 

\begin{figure}[!h]
\centering
\include{ReissnerMindlin_Case2_p=3_t=1e-2}
\caption{Case 2 (fully hardly simply supported annular plate). $p=3$, $t=10^{-2}$.} \label{fig:case2_results_1e-2}
\end{figure}

\begin{figure}[!h]
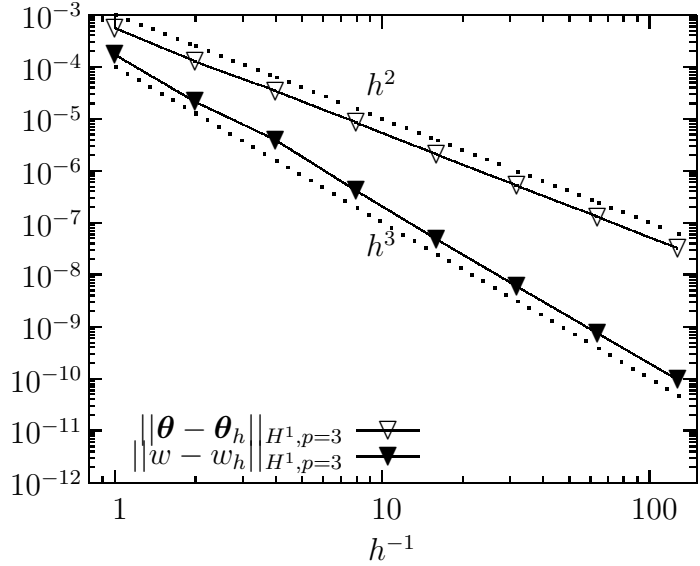

\centering
\include{ReissnerMindlin_Case2_p=3_t=1e-3}
\caption{Case 2 (fully hardly simply supported annular plate). $p=3$, $t=10^{-3}$.} \label{fig:case2_results_1e-3}
\end{figure}

\subsection{Case 3: boundary layer test problem}
The third test problem shares the same load and geometry used for Case~2 (Figure~\ref{fig:case2_geometry}), but with different boundary conditions. 
More precisely, we set $\Gamma_c = \emptyset$ and 

\begin{itemize}
	\item $\Gamma_{ss} =\{  (x,y)\in \R^2 \ : \  x^2+y^2=1 \}$; 
	\item $\Gamma_{sh}=\{ (0, y)\ :\ 1< y < 2.5  \}\cup \{(x,0) \ : \ 1< x< 2.5\} $   
	\item $\Gamma_{f} =\{ (x,y)\in \R^2  \ :\ x^2+y^2=\tfrac{25}{4}\}$.
\end{itemize}
On the curved parts of the boundary, this problem exhibits a boundary layer whose characteristic length is $O(t)$ (see~\cite{af2}).
In our computations we set $t=10^{-2}$, and $p=3$, $\alpha=2$ both for deflections and rotations. 
For this test case we always plot, in log-log scale, the errors versus $N_{DOF}^{1/2}$, where $N_{DOF}$ denots the total number of degrees of freedom.
We remark that for uniform meshes, $N_{DOF}^{1/2}$ behaves like $h^{-1}$.

As for Case~2, the analytical solution is not available, and a reference numerical solution is obtained using a very fine mesh.
%Errors on coarser meshes are again computed by means of that reference solution.

Figure~\ref{fig:case3_results_uniform_DOF} displays the error behaviour when using a uniform mesh refinement. It is worth noticing that  
a severe suboptimal convergence rate occurs as long as the mesh is not sufficiently fine to resolve the boundary layer. Of course, this phenomenon
implies that uniform meshes requires an excessive number of degrees of freedom to significantly reduce the error, especially for small 
plate thicknesses.  
%

%the behaviour of the approximation errors: for the coarsest meshes ($h \lesssim 1/64$) we have a very low error reduction, due to the fact that the layers are not well-resolved, then for the finest mesh
%$h \simeq 1/128$) we observe a stronger reduction in agreement with the fact that now the boundary layers are well resolved.

 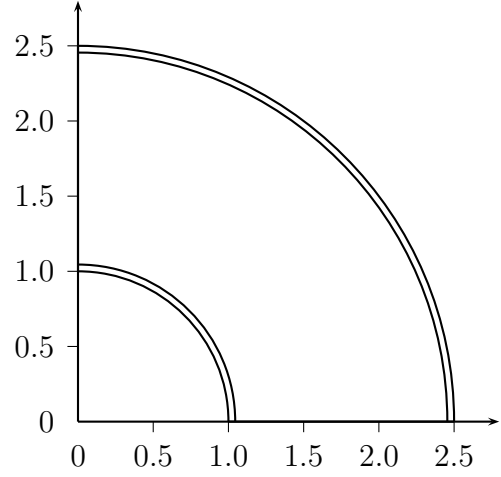
\begin{figure}[!h]
 	\centering
 	\psset{unit=2cm}
 	\begin{pspicture}(0.0,0.0)(3.0,3.0)
 		\psaxes[Dx=0.5,Dy=0.5,tickstyle=bottom]{->}(0.0,0.0)(2.8,2.8)
 		\psarc(0.0,0.0){1.0}{0}{90}
 		\psarc(0.0,0.0){1.045}{0}{90}
 		\psarc(0.0,0.0){2.455}{0}{90}
 		\psarc(0.0,0.0){2.5}{0}{90}
 		\psline(1.0,0.0)(2.5,0.0)
 		\psline(0.0,1.0)(0.0,2.5)
 	\end{pspicture}
 	\vskip0.5cm
 	\caption{Coarsest mesh for the layers-adapted case.} \label{fig:case3_geometry_adapted}
 \end{figure}

We now employ a sequence of meshes adapted to the boundary layers (Figure~\ref{fig:case3_geometry_adapted}), 
using the following procedure.
An initial mesh consisting of three elements is set up as in Figure~\ref{fig:case3_geometry_adapted}. The width of both the layer elements 
is 0.045 (corresponding to 3\% of the total annulus width). The finer meshes are then obtained by uniform refinement of each of the three initial elements.
Figure~\ref{fig:case3_results_adapted_DOF} shows the error behaviour using that mesh sequence. We notice that the optimal convergence rate 
for the $H^1$-norm error is now restored.

%\begin{figure}[!h]
%\centering
%\include{ReissnerMindlin_Case3_p=3_MeshUniform_t=1e-2}
%\caption{Case 3 (quarter of ring with boundary layers). Uniform mesh. $p=3$, $t=10^{-2}$} \label{fig:case3_results_uniform}
%\end{figure}

\begin{figure}[!h]
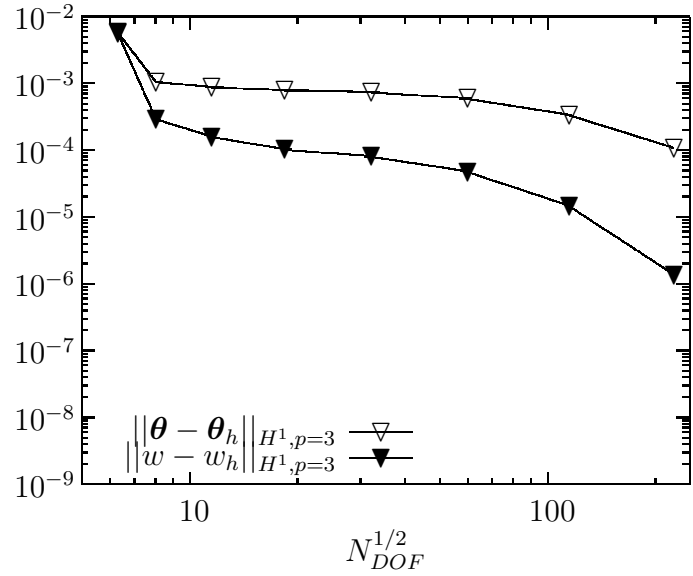

\centering
\include{ReissnerMindlin_Case3_p=3_MeshUniform_t=1e-2_DOF}
\caption{Case 3 (quarter of ring with boundary layers). Uniform mesh. $p=3$, $t=10^{-2}$} \label{fig:case3_results_uniform_DOF}
\end{figure}

%\begin{figure}[!h]
%\centering
%\include{ReissnerMindlin_Case3_p=3_MeshAdapted_layers=3e-2_t=1e-2}
%\caption{Case 3 (quarter of ring with boundary layers). Layers-adapted mesh. $p=3$, $t=10^{-2}$. $h$ represents the maximum element size in the parametric domain.} \label{fig:case3_results_adapted}
%\end{figure}

\begin{figure}[!h]
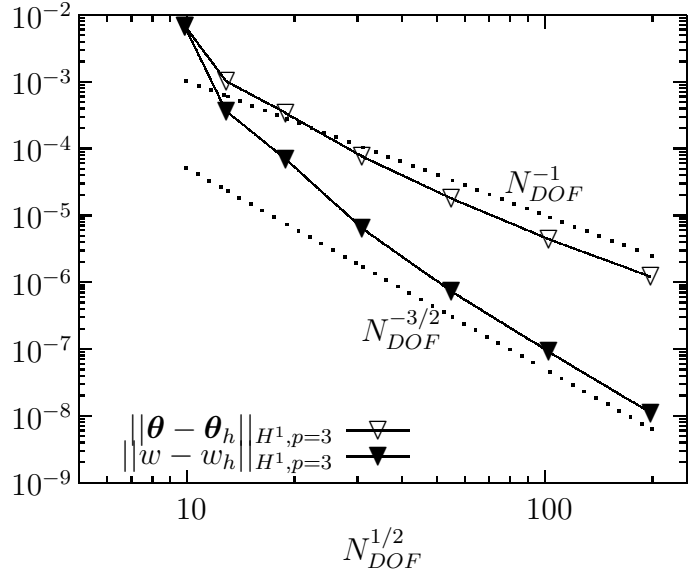

\centering
\include{ReissnerMindlin_Case3_p=3_MeshAdapted_layers=3e-2_t=1e-2_DOF}
\caption{Case 3 (quarter of ring with boundary layers). Layers-adapted mesh. $p=3$, $t=10^{-2}$.} \label{fig:case3_results_adapted_DOF}
\end{figure}

%\subsection{Case4: rigid motion patch test}

%\begin{figure}[!h]
%\centering
%\include{ReissnerMindlin_Case4_p=3_t=1e-2}
%\caption{Case 4 (quarter of ring with boundary conditions for rigid motion test). $p=3$, $t=10^{-2}$}
%\end{figure}

\section*{Acknowledgments}
The authors were partially supported by the European Research Council through the FP7 Ideas Starting Grant 205004: \emph{GeoPDEs - Innovative compatible discretization techniques for partial differential equations}, and by the Italian MIUR through the FIRB ``Futuro in Ricerca'' Grant RBFR08CZ0S \emph{Discretizzazioni Isogeometriche per la  Meccanica del Continuo}. Giancarlo Sangalli was also  partially supported by the European Research Council
Ideas Starting Grant 259229: \emph{ISOBIO - Isogeometric Methods for Biomechanics}. This support is gratefully acknowledged.

\newpage

%--------------------------------------------------------------------------
\bibliographystyle{model1-num-names}
\bibliography{Plate-IGA.bib}

\end{document}